\documentclass{amsart}

\usepackage[utf8]{inputenc} 
\usepackage[english]{babel} 
\usepackage{mathrsfs}

\theoremstyle{plain} 
\newtheorem*{thm}{Theorem} 

\begin{document} 
\title{An embedding constant for the Hardy space of Dirichlet series} 
\date{\today} 

\author{Ole Fredrik Brevig} \address{Department of Mathematical Sciences, Norwegian University of Science and Technology (NTNU), NO-7491 Trondheim, Norway} \email{ole.brevig@math.ntnu.no}

\thanks{The author is supported by Grant 227768 of the Research Council of Norway.}


\begin{abstract}
	A new and simple proof of the embedding of the Hardy--Hilbert space of Dirichlet series into a conformally invariant Hardy space of the half-plane is presented, and the optimal constant of the embedding is computed. 
\end{abstract}

\maketitle Let $\mathscr{H}^2$ denote the Hardy--Hilbert space of Dirichlet series, $f(s) = \sum_{n=1}^\infty a_n n^{-s}$, with square summable coefficients, and set
\[\|f\|_{\mathscr{H}^2} := \left(\sum_{n=1}^\infty |a_n|^2 \right)^\frac{1}{2}.\]
Using the Cauchy--Schwarz inequality, we find that a Dirichlet series $f\in\mathscr{H}^2$ is absolutely convergent in the half-plane $\mathbb{C}_{1/2} := \{s \,:\, \Re(s)>1/2\}$. To see that $\mathbb{C}_{1/2}$ is the largest half-plane of convergence for $\mathscr{H}^2$, consider $f(s) = \zeta(1/2+\varepsilon+s)$, where $\zeta$ denotes the Riemann zeta function and $\varepsilon>0$.

When studying function and operator theoretic properties of $\mathscr{H}^2$, it has proven fruitful to employ the embedding of $\mathscr{H}^2$ into the conformally invariant Hardy space of $\mathbb{C}_{1/2}$ (see e.g.~\cite[Sec.~9]{QS15}). The embedding inequality takes on the form 
\begin{equation}
	\label{eq:embedding} \|f\|_{H^2_{\operatorname{i}}} := \left(\frac{1}{\pi}\int_{-\infty}^\infty |f(1/2+it)|^2\,\frac{dt}{1+t^2}\right)^\frac{1}{2} \leq C\|f\|_{\mathscr{H}^2}. 
\end{equation}
Observe that the embedding inequality \eqref{eq:embedding} implies that Dirichlet series in $\mathscr{H}^2$ are locally $L^2$-integrable on the line $\Re(s)=1/2$. Indeed, the proofs of \eqref{eq:embedding} in the literature go via the local (but equivalent) formulation
\begin{equation}
	\label{eq:local} \sup_{\tau\in\mathbb{R}} \left(\int_{\tau}^{\tau+1} |f(1/2+it)|^2 \, dt\right)^\frac{1}{2} \leq \widetilde{C}\|f\|_{\mathscr{H}^2}. 
\end{equation}
To prove \eqref{eq:local}, one can use a general Hilbert--type inequality due to Montgomery and Vaughan \cite{MV74} or a version of the classical Plancherel--Polya inequality \cite[Thm.~4.11]{HLS97}. It is also possible to give Fourier analytic proofs of \eqref{eq:local}, the reader is referred to \cite[pp.~36--37]{OS12} and \cite[Sec.~1.4.4]{QQ13}. It should be pointed out that these proofs do not give a precise value for either of the constants $C$ and $\widetilde{C}$.

This note contains a new and simple proof of \eqref{eq:embedding}, which additionally identifies the optimal constant $C$. The proof is based on the observation that the $H^2_{\operatorname{i}}$-norm of a Dirichlet series is associated to a Hilbert--type bilinear form which is easy to estimate precisely.

\begin{thm}
	Suppose that $f(s) = \sum_{n=1}^\infty a_n n^{-s}$ is in $\mathscr{H}^2$. Then
	\[\left(\frac{1}{\pi}\int_{-\infty}^\infty \left|f\left(1/2+it\right)\right|^2\,\frac{dt}{1+t^2}\right)^\frac{1}{2} < \sqrt{2}\|f\|_{\mathscr{H}^2},\]
	and the constant $\sqrt{2}$ is optimal. 
\end{thm}
\begin{proof}
	Let $x$ be a positive real number. We begin by computing
	\[I(x) := \frac{1}{\pi} \int_{-\infty}^\infty x^{it}\,\frac{dt}{1+t^2} = \frac{1}{\pi} \int_{-\infty}^\infty \cos(|\log{x}|\,t)\,\frac{dt}{1+t^2} = e^{-|\log{x}|} = \frac{1}{\max(x,1/x)}.\]
	Expanding $|f(1/2+it)|^2$, we find that 
	\begin{equation}
		\label{eq:expanded} \|f\|_{H^2_{\operatorname{i}}}^2 = \sum_{m=1}^\infty \sum_{n=1}^\infty \frac{a_m \overline{a_n}}{\sqrt{mn}} I(n/m) = \sum_{m=1}^\infty \sum_{n=1}^\infty a_m \overline{a_n} \frac{\sqrt{mn}}{[\max(m,n)]^2}. 
	\end{equation}
	The identity \eqref{eq:expanded} will serve as the starting point for both the proof of the inequality $\|f\|_{H^2_{\operatorname{i}}}<\sqrt{2}\|f\|_{\mathscr{H}^2}$, and for the proof that $\sqrt{2}$ cannot be improved.
		
	Let us first consider the Hilbert--type (see \cite[Ch.~IX]{HLP}) bilinear form associated to \eqref{eq:expanded}. Given sequences $a,b\in\ell^2$, we want to estimate
	\[B(a,b) := \sum_{m=1}^\infty \sum_{n=1}^\infty a_m b_n \frac{\sqrt{mn}}{[\max(m,n)]^2}.\]
	By the Cauchy--Schwarz inequality, we find that
	\[|B(a,b)| \leq \left(\sum_{m=1}^\infty |a_m|^2 \sum_{n=1}^\infty \frac{m}{[\max(m,n)]^2}\right)^\frac{1}{2}\left(\sum_{n=1}^\infty |b_n|^2 \sum_{m=1}^\infty \frac{n}{[\max(m,n)]^2}\right)^\frac{1}{2}.\]
	Then $|B(a,b)|<2\|a\|_{\ell^2}\|b\|_{\ell^2}$, since
	\[\sum_{n=1}^\infty \frac{m}{[\max(m,n)]^2} = \sum_{n=1}^m \frac{m}{m^2} + \sum_{n=m+1}^\infty \frac{m}{n^2} < 1 + m\int_{m}^\infty \frac{dx}{x^2} = 2.\]
	Setting $b=\overline{a}$, we obtain the desired inequality $\|f\|_{H^2_{\operatorname{i}}}< \sqrt{2}\|f\|_{\mathscr{H}^2}$.
	
	For the optimality of $\sqrt{2}$, we again let $f(s) = \zeta(1/2+\varepsilon+s)$ for some $\varepsilon>0$. Clearly, $\|f\|_{\mathscr{H}^2}^2 = \zeta(1+2\varepsilon)$. We insert $f$ into \eqref{eq:expanded} and estimate the inner sums using integrals, which yields 
	\begin{align*}
		\|f\|_{H^2_{\operatorname{i}}}^2 &= \sum_{m=1}^\infty m^{-\varepsilon}\left(\frac{1}{m^2}\sum_{n=1}^m n^{-\varepsilon} + \sum_{n=m+1}^\infty\frac{ n^{-\varepsilon}}{n^2}\right) \\
		&> \sum_{m=1}^\infty m^{-\varepsilon}\left( \frac{1}{m^2}\frac{m^{1-\varepsilon}-1}{1-\varepsilon} + \frac{(m+1)^{-1-\varepsilon}}{1+\varepsilon}\right) \\
		&>\frac{\zeta(1+2\varepsilon)-\zeta(2+\varepsilon)}{1-\varepsilon}+\frac{\zeta(1+2\varepsilon)-1}{1+\varepsilon}. 
	\end{align*}
	Letting $\varepsilon\to0^+$, we conclude that if $\|f\|_{H^2_{\operatorname{i}}}\leq C\|f\|_{\mathscr{H}^2}$, then $C^2\geq2$.
\end{proof}

\bibliographystyle{amsplain} 
\bibliography{e} \end{document}